\documentclass[a4paper,12pt]{amsart}
\usepackage[utf8x]{inputenc}
\usepackage{amsmath,amssymb,amsthm}
\usepackage[all]{xy}
\usepackage{color}
\usepackage{hyperref}

\newtheorem{theorem}{Theorem}[section]

\newtheorem{lemma}[theorem]{Lemma}
\newtheorem{proposition}[theorem]{Proposition}

\theoremstyle{definition}

\theoremstyle{remark}
\newtheorem*{remark}{Remark}

\newcommand{\Ind}{{\rm Ind}}

\newcommand{\bfsig}{{\boldsymbol{\sigma}}}
\newcommand{\Gal}{\mathop{\textnormal{Gal}}}
\renewcommand{\phi}{\varphi}

\newcommand{\bsK}{{[\bfsig]_K}}
\newcommand{\bsKone}{{[\bfsig]_{K_1}}}

\begin{document}

\title{Random Galois extensions of Hilbertian fields}

\author{Lior Bary-Soroker \and Arno Fehm}

\begin{abstract}
Let $L$ be a Galois extension of a countable Hilbertian field $K$.
Although $L$ need not be Hilbertian, we prove that an abundance of
large Galois subextensions of $L/K$ are.
\end{abstract}

\maketitle

\section{Introduction}
\noindent
Hilbert's irreducibility theorem states that if $K$ is a number field and $f\in K[X,Y]$ is an irreducible polynomial
that is monic and separable in $Y$, then there exist infinitely many $a\in K$ such that $f(a,Y)\in K[Y]$ is irreducible.
Fields $K$ with this property are consequently called {\bf Hilbertian}, cf.~\cite{FriedJarden}, \cite{Lang}, \cite{Serre}.

Let $K$ be a field with a separable closure $K_s$, let $e\geq 1$, and write $\Gal(K)=\Gal(K_s/K)$ for the absolute Galois group of $K$. 
For an $e$-tuple $\bfsig = (\sigma_1, \ldots, \sigma_e) \in \Gal(K)^e$ we denote by 
\[
 [\bfsig]_K = \left< \sigma_\nu^\tau \mid \nu=1,\ldots,e\ \mbox{and} \ \tau\in \Gal(K)\right>
\]
the closed normal subgroup of $\Gal(K)$ that is generated by $\bfsig$.
For an algebraic extension $L/K$ we let
\[
L\bsK = \{a\in L \mid a^\tau = a, \ \forall \tau\in [\bfsig]_K\}
\]
be the maximal Galois subextension of $L/K$ that is fixed by each $\sigma_\nu$, $\nu=1,\ldots, e$. We note that the group $\bsK$, and hence the field $L\bsK$, depends on the base field $K$. 

Since $\Gal(K)^e$ is profinite, hence compact, it is equipped with a probability Haar measure. 
In \cite{Jarden} Jarden proves that if $K$ is countable and Hilbertian, then $K_s\bsK$ is Hilbertian for almost all $\bfsig\in \Gal(K)^e$. 
This provides a variety of large Hilbertian Galois extensions of $K$. 

Other fields of this type that were studied intensively are the fields $K_{{\rm tot},S}{\bsK}$, where $K$ is a number field, $S$ is a finite set of primes of $K$, and $K_{{\rm tot},S}$ is the 
{\bf field of totally $S$-adic numbers} over $K$ -- the maximal Galois extension of $K$ in which all primes in $S$ totally split; see for example \cite{HJPe} and the references therein for recent developments. Although the absolute Galois group of $K_{{\rm tot},S}{\bsK}$ was completely determined in {\it loc.\ cit.} (for almost all $\bfsig$), the question whether $K_{{\rm tot},S}{\bsK}$ is Hilbertian or not remained open. Note that if $\bfsig = (1,\ldots, 1)$, then  $K_{{\rm tot},S}{\bsK} =K_{{\rm tot},S}$ is not Hilbertian, cf.~\cite{BarySorokerFehm}.

The main objective of this study is to prove the following general result, which, in particular, generalizes Jarden's result and resolves the above question. 

\begin{theorem}\label{thm:main}
Let $K$ be a countable Hilbertian field, let $e\geq 1$, and let $L/K$ be a Galois extension. 
Then $L\bsK$ is Hilbertian for almost all $\bfsig \in \Gal(K)^e$. 
\end{theorem}

Jarden's proof of the case $L=K_s$ is based on Roquette's theorem \cite[Corollary~27.3.3]{FriedJarden} and Melnikov's theorem \cite[Theorem 25.7.5]{FriedJarden}: Jarden proves that for almost all $\bfsig$, the countable field $K_s\bsK$ is pseudo algebraically closed. Therefore, by Roquette, $K_s\bsK$ is Hilbertian if $[\bfsig]_K$ is a free profinite group of infinite rank. Then Melnikov's theorem is applied to reduce the proof of the freeness of $[\bfsig]_K$ to realizing simple groups as quotients of $[\bfsig]_K$. 

However, $L\bsK$ is not pseudo algebraically closed for most $L$ (e.g.\ for $L=K_{{\rm tot},S}$, whenever $S\neq \emptyset$). Thus, it seems that Jarden's proof cannot be extended to such fields $L$. 
Our proof utilizes Haran's twisted wreath products approach \cite{HaranDiamond}.
We can apply this approach whenever $L/K$ has many linearly disjoint subextensions (in the sense of Condition~\ref{condL} below).
A combinatorial argument then shows that in the remaining case, $L\bsK$ is a small extension of $K$, and therefore also Hilbertian.

\section{Small extensions and linearly disjoint families}
\noindent
Let $K\subseteq K_1\subseteq L$ be a tower of fields.
We say that $L/K_1$ satisfies {\bf Condition~\ref{condL}} if the following holds:
{\renewcommand{\theequation}{$\mathcal{L}_{K}$}
\begin{equation}\label{condL}
\begin{minipage}{11cm}
{\it There exists an infinite pairwise linearly disjoint family of finite proper subextensions of $L/K_1$ of the same degree and Galois over $K$.}
\end{minipage}
\end{equation}

If a Galois extension satisfies Condition~\ref{condL}, then one can find linearly disjoint families of subextensions with additional properties:}

\begin{lemma}\label{lem:lindisj}
Let $(M_i)_i$ be a pairwise linearly disjoint family of Galois extensions of $K$ and let $E/K$ be a finite Galois extension. Then $M_i$ is linearly disjoint from $E$ over $K$ for all but finitely many $i$. 
\end{lemma}

\begin{proof}
This is clear since $E/K$ has only finitely many subextensions, cf.~\cite[Lemma 2.5]{BarySorokerCharacterization} and its proof.
\end{proof}

\begin{lemma}\label{lem:Lplus}
Let $K\subseteq K_1\subseteq L$ be fields such that $L/K$ is Galois, $K_1/K$ is finite and $L/K_1$ satisfies Condition~\ref{condL}.
Let $M_0/K_1$ be a finite extension, and let $d\geq 1$. Then there exist
a finite group $G$ with $|G|\geq d$ and an infinite family $(M_i)_{i>0}$ of subextensions of $L/K_1$ which are Galois over $K$ such that $\Gal(M_i/K_1)\cong G$ for every $i>0$ and the family $(M_i)_{i\geq 0}$ is linearly disjoint over $K_1$.
\end{lemma}

\begin{proof}
By assumption there exists an infinite pairwise linearly disjoint family $(N_i)_{i>0}$ of subextensions of $L/K_1$ which are Galois over $K$ and of the same degree $n>1$ over $K_1$.
Iterating Lemma \ref{lem:lindisj} gives an infinite subfamily $(N_i')_{i>0}$ of $(N_i)_{i>0}$ such that the family $M_0,(N_i')_{i>0}$
is linearly disjoint over $K_1$.
If we let 
$$
M_i'=N_{id}'N_{id+1}'\cdots N_{id+d-1}'
$$ 
be the compositum, then the family
$M_0,(M_i')_{i>0}$ is linearly disjoint over $K_1$, and $[M_i':K_1]=n^d>d$ for every $i$.
Since up to isomorphism there are only finitely many finite groups of order $n^d$, 
there is a finite group $G$ of order $n^d$ and an infinite subfamily $(M_i)_{i>0}$ of $(M_i')_{i>0}$ such that $\Gal(M_i/K_1)\cong G$ for all $i>0$.
\end{proof}

\begin{lemma}\label{lem:Lup}
Let $K\subseteq K_1\subseteq K_2\subseteq L$ be fields such that $L/K$ is Galois, $K_2/K$ is finite Galois and $L/K_1$ satisfies Condition~\ref{condL}.
Then also $L/K_2$ satisfies Condition~\ref{condL}.
\end{lemma}

\begin{proof}
By Lemma \ref{lem:Lplus}, applied to $M_0=K_2$, there exists an infinite family $(M_i)_{i>0}$ of subextensions of $L/K_1$ which are Galois over $K$, of the same degree $n>1$ over $K_1$
and such that the family $K_2,(M_i)_{i>0}$ is linearly disjoint over $K_1$.
Let $M_i'=M_iK_2$. Then $[M_i':K_2]=[M_i:K_1]=n$, $M_i'/K$ is Galois,
and the family $(M_i')_{i>0}$ is linearly disjoint over $K_2$, cf.~\cite[Lemma~2.5.11]{FriedJarden}.
\end{proof}

Recall that a Galois extension $L/K$ is {\bf small} if for every $n\geq 1$ 
there exist only finitely many intermediate fields $K\subseteq M\subseteq L$ with $[M:K]=n$.
Small extensions are related to Condition~\ref{condL} by Proposition~\ref{lem:small} below,
for which we give a combinatorial argument using Ramsey's theorem,
which we recall for the reader's convenience:

\begin{proposition}[{\cite[Theorem 9.1]{Jech}}]\label{prop:Ramsey}
Let $X$ be a countably infinite set and $n,k\in\mathbb{N}$. 
For every partition $X^{[n]}=\bigcup_{i=1}^k Y_i$ of the set of subsets of $X$ of cardinality $n$ into $k$ pieces
there exists an infinite subset $Y\subseteq X$ such that $Y^{[n]}\subseteq Y_i$ for some $i$.
\end{proposition}

\begin{proposition}\label{lem:small}
Let $L/K$ be a Galois extension.
If there exists no finite Galois subextension $K_1$ of $L/K$ such that
$L/K_1$ satisfies Condition~\ref{condL}, then $L/K$ is small.
\end{proposition}

\begin{proof}
Suppose that $L/K$ is not small, so
it has infinitely many subextensions of degree $m$ over $K$, for some $m>1$.
Taking Galois closures we get that for some $1<d\leq m!$ there exists an infinite family $\mathcal{F}$ of Galois subextensions of $L/K$ of degree $d$: Indeed, only finitely many extensions of $K$ can have the same Galois closure.

Choose $d$ minimal with this property.
For any two distinct Galois subextensions of $L/K$ of degree $d$ over $K$ 
their intersection is a Galois subextension of $L/K$ of degree less than $d$ over $K$,
and by minimality of $d$ there are only finitely many of those.
Proposition \ref{prop:Ramsey} thus gives
a finite Galois subextension $K_1$ of $L/K$ and an infinite subfamily $\mathcal{F}'\subseteq\mathcal{F}$ such that for any two distinct
$M_1,M_2\in\mathcal{F}'$, $M_1\cap M_2=K_1$.
Since any two Galois extensions are linearly disjoint over their intersection, it follows that $L/K_1$ satisfies Condition~\ref{condL}.
\end{proof}

The converse of Proposition~\ref{lem:small} holds trivially.
The following fact on small extensions will be used in the proof of Theorem~\ref{thm:main}.

\begin{proposition}[{\cite[Proposition 16.11.1]{FriedJarden}}]\label{lem:smallext}
If $K$ is Hilbertian and $L/K$ is a small Galois extension, then $L$ is Hilbertian.
\end{proposition}

\section{Measure theory}
\noindent
For a profinite group $G$ we denote by $\mu_G$ the probability Haar measure on $G$. 
We will make use of the following two very basic measure theoretic facts.

\begin{lemma}\label{lem:indep}
Let $G$ be a profinite group, $H\leq G$ an open subgroup, $\Sigma_1,\dots,\Sigma_k\subseteq H$ measurable $\mu_H$-independent sets,
and $S\subseteq G$ a set of representatives of $G/H$.

Let $\Sigma_i^*=\bigcup_{g\in S}g\Sigma_i$.
Then $\Sigma_1^*,\dots,\Sigma_k^*$ are $\mu_G$-independent.
\end{lemma}

\begin{proof}
Let $n=[G:H]$. Then for any measurable $X\subseteq H$ we have $\mu_H(X)=n\mu_G(X)$.
Since $G$ is the disjoint union of the cosets $gH$, for $g\in S$,
we have that
$$
\mu_G(\Sigma_i^*)=\sum_{g\in S} \mu_G(g\Sigma_i)=n\mu_G(\Sigma_i)=\mu_H(\Sigma_i)
$$
and
\begin{eqnarray*}
\mu_G\left(\bigcap_{i=1}^k\Sigma_i^*\right)&=&\sum_{g\in S}\mu_G\left(\bigcap_{i=1}^k g\Sigma_i\right)=n\mu_G\left(\bigcap_{i=1}^k\Sigma_i\right)=\\
&=&\mu_H\left(\bigcap_{i=1}^k\Sigma_i\right)=
 \prod_{i=1}^k\mu_H\left(\Sigma_i\right)=\prod_{i=1}^k\mu_G\left(\Sigma_i^*\right),
\end{eqnarray*}
thus $\Sigma_1^*, \ldots, \Sigma_k^*$ are $\mu_G$-independent.
\end{proof}

\begin{lemma}\label{lem:measure}
Let $(\Omega,\mu)$ be a measure space. For each $i\geq 1$ let $A_i\subseteq B_i$ be measurable subsets of $\Omega$.
If $\mu(A_i)=\mu(B_i)$ for every $i\geq 1$, then $\mu(\bigcup_{i=1}^\infty A_i)=\mu(\bigcup_{i=1}^\infty B_i)$.
\end{lemma}

\begin{proof}
This is clear since
$$
 \left(\bigcup_{i=1}^\infty B_i\right)\smallsetminus\left(\bigcup_{i=1}^\infty A_i\right)\subseteq\bigcup_{i=1}^\infty (B_i\smallsetminus A_i),
$$ 
and $\mu(B_i\smallsetminus A_i)=0$ for every $i\geq 1$ by assumption.
\end{proof}

\section{Twisted wreath products}
\noindent
Let $A$ and $G_1\leq G$ be finite groups together with a (right) action of $G_1$ on $A$. 
The set of $G_1$-invariant functions from $G$ to $A$,
$$
 {\rm Ind}_{G_1}^G (A) = \left\{ f\colon G\to A \;\mid\; f(\sigma\tau)=f(\sigma)^\tau,\ \forall \sigma\in G \,\forall\tau\in G_1\right\},
$$ 
forms a group under pointwise multiplication. Note that ${\rm Ind}_{G_1}^G (A) \cong A^{[G:G_1]}$. The group $G$ acts on $ {\rm Ind}_{G_1}^G (A) $ from the right
by $f^\sigma(\tau)= f(\sigma\tau)$, for all $\sigma,\tau\in G$. The \textbf{twisted wreath product} is defined to be the semidirect product
\[
 A\wr_{G_1} G = {\rm Ind}_{G_1}^G(A) \rtimes G,
\]
cf.~\cite[Definition 13.7.2]{FriedJarden}.
Let $\pi\colon {\rm Ind}_{G_1}^G (A) \to A$ be the projection given by $\pi(f) = f(1)$. 

\begin{lemma}\label{lem:product}
Let $G = G_1\times G_2$ be a direct product of finite groups, let $A$ be a finite $G_1$-group, and let $I={\rm Ind}_{G_1}^G (A)$. Assume that $|G_2| \geq  |A|$. Then there exists $\zeta\in  I$ such that for every $g_1\in G_1$,
the normal subgroup $N$ of $A\wr_{G_1} G$ generated by $\tau = (\zeta, (g_1,1))$ satisfies $\pi(N\cap I)=A$. 
\end{lemma}

\begin{proof}
Let $A = \{ a_1, \ldots, a_n\}$ with $a_1=1$. By assumption, $|G_2| \geq n$, so we may choose distinct elements $h_1 , \ldots, h_n\in G_2$ with $h_1=1$. For $(g,h)\in G$ we set
\[
\zeta(g,h) = \begin{cases}
a_i^{g},& \mbox{if }h=h_i\ \mbox{ for some } i\\
1, & \mbox{otherwise.}\\

\end{cases}
\]
Then  $\zeta\in I$.
Since $G_1$ and $G_2$ commute in $G$, for any $h\in G_2$ we have 
\[
\tau \tau^{-h} = \zeta g_1(\zeta g_1)^{-h} = \zeta g_1\cdot g_1^{-1}\zeta^{-h} = \zeta\zeta^{-h} \in N\cap I.
\]
Hence,
\begin{eqnarray*}
a_{i}^{-1} &=&a_{1}a_{i}^{-1} = \zeta(1)\zeta(h_i)^{-1} =(\zeta\zeta^{-h_i})(1)\\
&=&(\tau\tau^{-h_i})(1)= \pi(\tau\tau^{-h_i}) \in \pi(N\cap I).
\end{eqnarray*}
We thus conclude that $A= \pi(N\cap I)$, as claimed.
\end{proof}

Following \cite{HaranDiamond} we say that
a tower of fields  
$$
 K\subseteq E'\subseteq  E\subseteq  N\subseteq \hat{N}
$$ 
\textbf{realizes} a twisted wreath product $A\wr_{G_1} G$ if 
$\hat{N}/K$ is a Galois extension with Galois group isomorphic to $A\wr_{G_1} G$ and the tower of fields corresponds to the subgroup series
\[
A\wr_{G_1} G \;\geq\; \Ind_{G_1}^G(A) \rtimes G_1 \;\geq\; \Ind_{G_1}^G(A) \;\geq\; {\rm ker}(\pi)\;\geq\; 1.
\]
In particular we have the following commutative diagram:
\[
\xymatrix{
\Gal(\hat{N}/E) \ar[r]^\cong\ar[d]^{{\rm res}} &\Ind_{G_1}^G(A)\ar[d]^\pi \\
 \Gal(N/E)\ar[r]^\cong & A.
}
\]

\section{Hilbertian fields}
\noindent
We will use the following specialization result for Hilbertian fields:

\begin{lemma}\label{lem:places}
Let $K_1$ be a Hilbertian field, 
let $\mathbf{x}=(x_1, \ldots, x_d)$ be a finite tuple of variables, let  $0\neq g(\mathbf{x})\in K_1[\mathbf{x}]$,
and consider field extensions $M,E,E_1,N$ of $K_1$ as in the following diagram.
\[
\xymatrix{
	&M\ar@{-}[r]
		&ME_1\ar@{-}[r]
			&ME_1(\mathbf{x})\ar@{-}[r]
				&MN
\\
K_1\ar@{-}[r]
	&E\ar@{-}[r]\ar@{-}[u]
		&E_1\ar@{-}[u]\ar@{-}[r]
			&E_1(\mathbf{x})\ar@{-}[u]\ar@{-}[r]
				&N\ar@{-}[u]
}
\]
Assume that $E, E_1, M$ are finite Galois extensions of $K_1$, $E= E_1\cap M$, 
$N$ is a finite Galois extension of $K_1(\mathbf{x})$ that is regular over $E_1$, 
and let $y\in N$. 
Then there exists an $E_1$-place $\phi$ of $N$ such that $\boldsymbol{b}=\phi(\mathbf{x})$ and $\phi(y)$ are finite, $g(\boldsymbol{b})\neq 0$, the residue fields of $K_1(\mathbf{x})$, $E_1(\mathbf{x},y)$ and $N$ are $K_1$, $E_1(\phi(y))$ and $\bar{N}$, respectively, where $\bar{N}$ is a Galois extension of $K_1$ which is linearly disjoint from $M$ over $E$, and $\Gal(\bar{N}/K_1)\cong \Gal(N/K_1(\mathbf{x}))$.
\end{lemma}

\begin{proof}
$E_1$ and $M$ are linearly disjoint over $E$, and $N$ and $ME_1$ are linearly disjoint over $E_1$. We thus get that $M$ and $N$ are linearly disjoint over $E$. Thus $N$ is linearly disjoint from $M(\mathbf{x})$ over $E(\mathbf{x})$, so $N\cap M(\mathbf{x})=E(\mathbf{x})$.

For every $\mathbf{b}\in K_1^d$ there exists a $K_1$-place $\varphi_{\mathbf{b}}$ of $K_1(\mathbf{x})$ with residue field $K_1$ and $\varphi_{\mathbf{b}}(\mathbf{x})=\mathbf{b}$. It extends uniquely to $ME_1(\mathbf{x})$, and the residue fields of
$M(\mathbf{x})$ and $E_1(\mathbf{x})$ are $M$ and $E_1$, respectively.

Since $K_1$ is Hilbertian, by \cite[Lemma 13.1.1]{FriedJarden} (applied to the three separable extensions $E_1(\mathbf{x},y)$, $N$ and $MN$ of $K_1(\mathbf{x})$)
there exists $\mathbf{b}\in K_1^d$ with $g(\mathbf{b})\neq 0$ such that any extension $\varphi$ of $\varphi_{\mathbf{b}}$ to $MN$ satisfies the following: 
$\phi(y)$ is finite, the residue field of $E_1(\mathbf{x},y)$ is $E_1(\varphi(y))$, 
the residue fields $\overline{MN}$ and $\overline{N}$ of $MN$ and $N$, respectively, are Galois over $K_1$, 
and $\varphi$ induces isomorphisms $\Gal(N/K_1(\mathbf{x}))\cong\Gal(\overline{N}/K_1)$ and $\Gal(MN/K_1(\mathbf{x}))\cong\Gal(\overline{MN}/K_1)$.

By Galois correspondence, the latter isomorphism induces an isomorphism of the lattices of intermediate fields of $MN/K_1(\mathbf{x})$
and $\overline{MN}/K_1$. Hence, $N\cap M(\mathbf{x})=E(\mathbf{x})$ implies that
$\overline{N}\cap M=E$, which means that $\overline{N}$ and $M$ are linearly disjoint over $E$.
\end{proof}

We will apply the following Hilbertianity criterion:

\begin{proposition}[{\cite[Lemma 2.4]{HaranDiamond}}]\label{prop:haran}
Let $P$ be a field and let $x$ be transcendental over $P$.
Then $P$ is Hilbertian if and only if for every absolutely irreducible $f\in P[X,Y]$, monic in $Y$, and every
finite Galois extension $P'$ of $P$ such that $f(x,Y)$ is Galois over $P'(x)$, there are
infinitely many $a\in P$ such that $f(a,Y)\in P[Y]$ is irreducible over $P'$.
\end{proposition}

\section{Proof of Theorem \ref{thm:main}}

\begin{lemma}\label{lem:main}
Let $K\subseteq K_1\subseteq L$ be fields such that $K$ is Hilbertian, $L/K$ is Galois, $K_1/K$ is finite Galois, and $L/K_1$ satisfies Condition~\ref{condL}. 
Let $e\geq 1$, let $f\in K_1[X,Y]$ be an absolutely irreducible polynomial that is Galois over $K_s(X)$ and let $K_1'$ be a finite separable extension of $K_1$. Then for almost all $\bfsig \in \Gal(K_1)^e$ there exist infinitely many $a \in L\bsK$ such that $f(a,Y)$ is irreducible over $K_1'\cdot L\bsK$.
\end{lemma}

\begin{proof}
Let $E$ be a finite Galois extension of $K$ such that $K_1'\subseteq E$ and $f$ is Galois over $E(X)$ and put $G_1=\Gal(E/K_1)$. 
Let $x$ be  transcendental over $K$ and $y$ such that $f(x,y)=0$. Let $F'=K_1(x,y)$ and $F=E(x,y)$. Since $f(X,Y)$ is absolutely irreducible, $F'/K_1$ is regular, hence $\Gal(F/F')\cong G_1$. Since $f(X,Y)$ is Galois over $E(X)$, $F/K_1(x)$ is Galois (as the compositum of $E$ and the splitting field of $f(x,Y)$ over $K_1(x)$). Then $A=\Gal(F/E(x))$ is a subgroup of 
$\Gal(F/K_1(x))$,
so 
$G_1=\Gal(F/F')$ acts on $A$ 
by conjugation. 
\[
\xymatrix{
F'\ar@{-}[r]^{G_1}&F\\
K_1(x)\ar@{-}[u]\ar@{-}[r]^{G_1}& E(x)\ar@{-}[u]_A
}
\]
Since $L/K_1$ satisfies Condition~\ref{condL},
by Lemma \ref{lem:Lplus}, applied to $M_0=E$, there exists a finite group $G_2$ with $d:=|G_2|\geq|A|$ and a sequence $(E_{i}')_{i>0}$ of linearly disjoint subextensions of $L/K_1$ which are Galois over $K$ with $\Gal(E_{i}'/K_1)\cong G_2$ such that the family $E,(E_i')_{i>0}$ is linearly disjoint over $K_1$.
Let $E_i = EE_{i}'$. Then $E_i/K$ is Galois and $\Gal(E_i/K_1)\cong G:=G_1\times G_2$ for every $i$.

Let $\mathbf{x} = (x_{1}, \ldots, x_{d})$ be a $d$-tuple of variables, and for each $i$ 
choose a basis $w_{i1}, \ldots, w_{id}$ of $E_{i}'/K_1$.
By \cite[Lemma 3.1]{HaranDiamond}, for each $i$ we have a tower 
\begin{equation}\label{tower}
K_1(\mathbf{x})\subseteq E_{i}'(\mathbf{x}) \subseteq E_{i}(\mathbf{x}) \subseteq N_i \subseteq \hat{N}_i
\end{equation}
that realizes the twisted wreath product $A\wr_{G_1} G$, such that $\hat{N}_i$ is regular over $E_i$ and $N_i = E_{i}(\mathbf{x})(y_i)$, where ${\rm irr}(y_i, E_{i}(\mathbf{x})) = f(\sum_{\nu=1}^{d}  w_{i\nu} x_{\nu}, Y)$.

We inductively construct an ascending sequence $(i_j)_{j=1}^\infty$ of positive integers 
and for each $j\geq 1$ an $E_{i_j}$-place $\phi_j$ of $\hat{N}_{i_j}$ such that
\begin{enumerate}
\item[(a)] the elements $a_j := \sum_{\nu=1}^d w_{i_j \nu}\phi_j(x_{\nu})\in E_{i_j}'$ are distinct for $j\geq 1$,
\item[(b)] the residue field tower of \eqref{tower}, for $i=i_j$, under $\phi_j$,
\begin{equation}\label{tower2}
K_1\subseteq E_{i_j}' \subseteq E_{i_j} \subseteq M_{i_j}\subseteq \hat{M}_{i_j},
\end{equation}
realizes the twisted wreath product $A\wr_{G_1} G$ and $M_{i_j}$ is generated by a root of $f(a_j, Y)$ over $E_{i_j}$,
\item[(c)] the family $(\hat{M}_{i_j})_{j=1}^\infty$ is linearly disjoint over $E$.
\end{enumerate}
Indeed, suppose that $i_1,\dots,i_{j-1}$ and $\phi_1,\dots,\phi_{j-1}$ are already constructed and let $M = \hat{M}_{i_1} \cdots\hat{M}_{i_{j-1}}$. 
By Lemma~\ref{lem:lindisj} there is $i_{j}>i_{j-1}$ such that $E_{i_{j}}'$ is linearly disjoint from $M$ over $K_1$. Thus, $E_{i_j}$ is linearly disjoint from $M$ over $E$.  
Since $K$ is Hilbertian and $K_1/K$ is finite, $K_1$ is Hilbertian.
Applying Lemma~\ref{lem:places} to $M$, 
$E$, $E_{i_j}$, $\hat{N}_{i_j}$, and  $y_{i_j}$, gives an $E_{i_j}$-place $\phi_j$ of $\hat{N}_{i_j}$ such that (b) and (c) are satisfied. 
Choosing $g$ suitably we may assume that $a_j=\phi_j(\sum_{\nu=1}^d w_{i_j\nu}x_\nu)\notin\{a_1,\dots,a_{j-1}\}$, so also (a) is satisfied.

We now fix $j$ and make the following identifications:
$\Gal(\hat{M}_{i_j}/K_1)=A\wr_{G_1}G=I\rtimes (G_1\times G_2)$,
$\Gal(\hat{M}_{i_j}/E_{i_j})=I$,
$\Gal(M_{i_j}/E_{i_j})=A$.
The restriction map $\Gal(\hat{M}_{i_j}/E_{i_j})\rightarrow\Gal(M_{i_j}/E_{i_j})$ 
is thus identified with $\pi:A\wr_{G_1}G\rightarrow A$,
and $\Gal(\hat{M}_{i_j}/M_{i_j})=\ker(\pi)$.
Let $\zeta\in I:=\Ind_{G_1}^{G}(A)$ be as in Lemma \ref{lem:product}
and let $\Sigma_j$ be the set of those $\bfsig\in\Gal(K_1)^e$ such that 
for every $\nu\in\{1,\dots,e\}$,
$\sigma_\nu|_{\hat{M}_{i_j}}=(\zeta,(g_{\nu1},1))\in I\rtimes (G_1\times G_2)$ for some $g_{\nu1}\in G_1$. Then the normal subgroup $N$ generated by $\bfsig|_{\hat{M}_{i_j}}$ in $\Gal(\hat{M}_{i_j}/K_1)$ satisfies $\pi(N\cap I) = A$. 

Now fix $\bfsig=(\sigma_1, \ldots, \sigma_e)\in\Sigma_j$ and let $P=L\bsK$ and $Q=K_s{[\bfsig]_{K_1}}$.
Then 
$$
 P = L\cap K_s\bsK \subseteq K_s\bsK \subseteq K_s{[\bfsig]_{K_1}} = Q.
$$ 
Since $E_{i_j}'$ is fixed by $\sigma_\nu$, $\nu=1, \ldots, e$, and Galois over $K$, we have $E_{i_j}'\subseteq P\subseteq Q$. Thus  $a_{j}\in  P$ and $E_{i_j}Q=EQ$. 
Therefore, since $M_{i_j}$ is generated by a root of $f(a_j,Y)$ over $E_{i_j}$, we get that $M_{i_j}Q$ is generated by a root of $f(a_j, Y)$ over $EQ$.

\[
\xymatrix{
Q\ar@{-}[r]\ar@{-}[d]
	&E_{i_j}Q\ar@{-}[r]\ar@{-}[d]
		&M_{i_j}Q\ar@{-}[r]\ar@{-}[d]
			&\hat{M}_{i_j}Q\ar@{-}[d]\\
\hat{M}_{i_j}\cap Q\ar@{-}[r]
	&(\hat{M}_{i_j}\cap Q)E_{i_j}\ar@{-}[r]\ar@{-}[d]
		&(\hat{M}_{i_j}\cap Q)M_{i_j}\ar@{-}[r]\ar@{-}[d]
			&\hat{M}_{i_j}\ar@{.}[ld]|{\ker(\pi)}\ar@{.}[lld]|{I}\ar@/^15pt/@{.}[lll]|N\\
	&E_{i_j}\ar@{-}[r]_A
		&M_{i_j}
}
\]
The equality $N=\Gal(\hat{M}_{i_j}/\hat{M}_{i_j}\cap Q)$ gives
\[
\Gal(\hat{M}_{i_j}Q/M_{i_j}Q) \cong \Gal(\hat{M}_{i_j}/(\hat{M}_{i_j}\cap Q)M_{i_j}) = N\cap\ker(\pi)
\]
and
\[
\Gal(\hat{M}_{i_j}Q/E_{i_j}Q) \cong \Gal(\hat{M}_{i_j}/(\hat{M}_{i_j}\cap Q)E_{i_j}) = N\cap I.
\]
Therefore,
\begin{eqnarray*}
&\Gal(M_{i_j}Q/E_{i_j}Q) \cong (N\cap I)/(N\cap\ker(\pi)) \cong \pi(N\cap I) = A.
\end{eqnarray*}
Since $|A| = \deg_Y f(X,Y) = \deg f(a_j,Y)$, we get that $f(a_j,Y)$ is irreducible over $EQ$.
Finally, we have 
$K_1'P\subseteq EP \subseteq EQ$, therefore $f(a_j,Y)$ is irreducible over $K_1'P$.

It suffices to show that almost all $\bfsig\in \Gal(K_1)^e$ lie in infinitely many $\Sigma_j$.
Since, by (c), the family $(\hat{M}_{i_j})_{j=1}^\infty$ is linearly disjoint over $E$, the sets $\Sigma_{j}$ are
independent for $\mu=\mu_{\Gal(K_1)^e}$ (Lemma~\ref{lem:indep}).
Moreover, 
\[
\mu(\Sigma_{j})=\frac{|G_1|^e}{|A\wr_{G_1}G|^e}>0
\] 
does not depend on $j$, so $\sum_{j=1}^\infty\mu(\Sigma_{j})=\infty$.
It follows from the Borel-Cantelli lemma \cite[Lemma 18.3.5]{FriedJarden} that almost all $\bfsig\in\Gal(K_1)^e$ lie in  infinitely many $\bfsig\in\Sigma_{j}$.
\end{proof}

\begin{proposition}\label{prop:main}
Let $K\subseteq K_1\subseteq L$ be fields such that $K$ is countable Hilbertian, $L/K$ is Galois, $K_1/K$ is finite Galois and $L/K_1$ satisfies Condition~\ref{condL}.
Let $e\geq 1$.
Then $L\bsK$ is Hilbertian for almost all $\bfsig\in\Gal(K_1)^e$.
\end{proposition}

\begin{proof}
Let $\mathcal{F}$ be the set of all triples $(K_2,K_2',f)$, where $K_2$ is a finite subextension of $L/K_1$ which is Galois over $K$, $K_2'/K_2$ is a finite separable extension (inside a fixed separable closure $L_s$ of $L$), and $f(X,Y)\in K_2[X,Y]$ is an absolutely irreducible polynomial that is Galois over $K_s(X)$. 
Since $K$ is countable, the family $\mathcal{F}$ is also countable.
If $(K_2,K_2',f)\in\mathcal{F}$, then $K_2$ is Hilbertian (\cite[Corollary 12.2.3]{FriedJarden}) and 
$L/K_2$ satisfies Condition~\ref{condL} (Lemma \ref{lem:Lup}),
hence Lemma \ref{lem:main} gives a set $\Sigma_{(K_2,K_2',f)}'\subseteq\Gal(K_2)^e$ of full measure in $\Gal(K_2)^e$
such that for every $\bfsig\in\Sigma'_{(K_2,K_2',f)}$ there
exist infinitely many $a\in L\bsK$ such that $f(a,Y)$ is irreducible over $K_2'\cdot L\bsK$. 
Let 
\[
 \Sigma_{(K_2,K_2',f)}=\Sigma_{(K_2,K_2',f)}'\cup (\Gal(K_1)^e\smallsetminus\Gal(K_2)^e).
\] 
Then $\Sigma_{(K_2,K_2',f)}$ has measure $1$ in $\Gal(K_1)^e$. We conclude that the measure of  $\Sigma=\bigcap_{(K_2,K_2',f)\in\mathcal{F}}\Sigma_{(K_2,K_2',f)}$ is $1$.

Fix a $\bfsig\in\Sigma$ and let $P=L\bsK$. Let $f\in P[X,Y]$ be absolutely irreducible and monic in $Y$, and let $P'$ be a finite Galois extension of $P$ such that $f(X,Y)$ is Galois over $P'(X)$. In particular, $f$ is Galois over $K_s(X)$. Choose a finite extension $K_2/K_1$ which is Galois over $K$ such that $K_2\subseteq P\subseteq L$ and $f\in K_2[X,Y]$. Let $K_2'$ be a finite extension of $K_2$ such that $PK_2'=P'$. Then $\bfsig\in \Gal(K_2)^{e}$. Since, in addition, $\bfsig\in \Sigma_{(K_2,K_2',f)}$, we get that $\bfsig\in \Sigma_{(K_2,K_2',f)}'$. Thus there exist infinitely many $a\in P$ such that $f(a,Y)$ is irreducible over $PK_2'=P'$. So,  by Proposition~\ref{prop:haran}, $P$ is Hilbertian.
\end{proof}

\begin{remark}
The proof of Proposition~\ref{prop:main} actually gives a stronger assertion: Under the assumptions of the proposition, for almost all $\bfsig \in \Gal(K_1)^e$ the field $K_s\bsKone$ is Hilbertian over $L\bsK$ in the sense of \cite[Definition 7.2]{BSPAC}. 
In particular, if $L/K$ satisfies Condition~\ref{condL}
(this holds for example for $L=K_{{\rm tot},S}$ from the introduction), then $K_s\bsK$ is Hilbertian over $L\bsK$. Since this is not the objective of this work, the details are left as an exercise for the interested reader. 
\end{remark}

\begin{proof}[Proof of Theorem \ref{thm:main}]
Let $K$ be a countable Hilbertian field, let $e\geq 1$, and let $L/K$ be a Galois extension. We need to prove that $L\bsK$ is Hilbertian for almost all $\bfsig\in \Gal(K)^e$. 

Let $\mathcal{F}$ be the set of finite Galois subextensions $K_1$ of $L/K$
for which $L/K_1$ satisfies Condition~\ref{condL}. 
Note that $\mathcal{F}$ is countable, since $K$ is.

Let $\Omega = \Gal(K)^e$, let $\mu=\mu_\Omega$, and let 
$$
 \Sigma=\{\bfsig\in\Omega : L\bsK \mbox{ is Hilbertian}\}.
$$ 
For $K_1\in\mathcal{F}$ let $\Omega_{K_1}=\Gal(K_1)^e$ and $\Sigma_{K_1}=\Omega_{K_1}\cap \Sigma$. Note that
$$
 \Omega_{K_1} = \left\{\bfsig\in\Omega : K_1\subseteq L\bsK \right\}.
$$
By Proposition~\ref{prop:main}, $\mu(\Sigma_{K_1})=\mu(\Omega_{K_1})$ for each $K_1$.
Let
$$
 \Delta \;:=\; \Omega \smallsetminus \bigcup_{K_1\in\mathcal{F}} \Omega_{K_1} \;=\; \left\{ \bfsig\in\Omega: K_1\not\subseteq L\bsK\mbox{ for all }K_1\in\mathcal{F}\right\}.
$$
If $\bfsig\in\Delta$, then $L\bsK/K$ is small by Proposition~\ref{lem:small}, so $L\bsK$ is Hilbertian by Proposition~\ref{lem:smallext}. 
Thus, $\Delta\subseteq\Sigma$. 
Since $\Omega = \Delta\cup\bigcup_{K_1\in\mathcal{F}} \Omega_{K_1}$, Lemma~\ref{lem:measure} implies that
$$
 \mu(\Sigma) = \mu\left((\Sigma\cap \Delta)\cup\bigcup_{K_1\in\mathcal{F}} \Sigma_{K_1}\right) = \mu\left(\Delta\cup\bigcup_{K_1\in\mathcal{F}} \Omega_{K_1}\right)=\mu(\Omega)=1,
$$
which concludes the proof of the theorem.
\end{proof}

\section*{Acknowledgements}
\noindent
The authors would like to express their sincere thanks to Moshe Jarden for pointing out a subtle gap in a previous version and for many useful suggestions and remarks.
This research was supported by the Lion Foundation Konstanz / Tel Aviv and the Alexander von Humboldt Foundation. 

\bibliographystyle{plain}
\bibliography{lit}

\end{document}